\documentclass[11pt,reqno]{amsart}
\usepackage{amssymb,amsmath}
\usepackage{amsthm}
\usepackage{color,graphicx}
\usepackage{hyperref}

\newcommand{\sech}{{\rm \,sech}}
\newcommand{\R}{{\mathbb R}}
\newcommand{\Lum}{{\mathcal{L}_1}}
\newcommand{\Ldois}{{\mathcal{L}_2}}

\newcommand{\I}{{\mathcal{I}}}
\newcommand{\Z}{{\mathcal{Z}}}
\newcommand{\RR}{{\mathcal{R}}}
\newcommand{\ve}{{\varepsilon}}
\newcommand{\Ll}{\mathbb{L}}
\newcommand{\Hh}{\mathbb{H}}
\newcommand{\vf}{{\varphi}}
\newcommand{\ra}{{\rangle}}
\newcommand{\la}{{\langle}}

\numberwithin{equation}{section}

\newtheorem{theorem}{Theorem}[section]
\newtheorem{proposition}[theorem]{Proposition}
\newtheorem{remark}[theorem]{Remark}
\newtheorem{lemma}[theorem]{Lemma}
\newtheorem{corollary}[theorem]{Corollary}
\newtheorem{definition}[theorem]{Definition}

\begin{document}
\vglue-1cm \hskip1cm
\title[Fourth-order dispersive NLS equation]{The Fourth-order dispersive nonlinear Schr\"odinger equation: Orbital stability of a standing wave}


\maketitle

\begin{center}
{\bf F\'abio Natali}

{Departamento de Matem\'atica - Universidade Estadual de Maring\'a\\
Avenida Colombo, 5790, CEP 87020-900, Maring\'a, PR, Brazil.}\\
{ fmnatali@hotmail.com}

\vspace{3mm}

{\bf Ademir Pastor}

{ IMECC-UNICAMP\\
Rua S\'ergio Buarque de Holanda, 651, CEP 13083-859, Campinas, SP,
Brazil.  } \\{ apastor@ime.unicamp.br}
\end{center}

\begin{abstract}
Considered in this report is the one-dimensional fourth-order dispersive cubic nonlinear Schr\"odinger equation with mixed dispersion. Orbital stability, in the energy space, of a particular standing-wave solution is proved in the context of Hamiltonian systems. The main result is established by constructing a suitable Lyapunov function.
\end{abstract}

\section{Introduction}

Considered here is the cubic fourth-order Schr\"odinger equation with mixed dispersion
\begin{equation}\label{nls}
iu_t+u_{xx}-u_{xxxx}+|u|^2u=0,
\end{equation}
where $x,t\in\R$ and $u=u(x,t)$ is a complex-valued function. Equation \eqref{nls} was introduced in \cite{kar} and \cite{kar1} and it appears in the propagation of intense laser beams in a bulk medium with Kerr nonlinearity when small fourth-order
dispersion are taking into account. In addition, equation $(\ref{nls})$ has been considered in connection
with the nonlinear fiber optics and the theory of optical solitons in gyrotropic media. From the mathematical point of view, \eqref{nls} brings some interesting questions because it does not enjoy scaling invariance.

The $n$-dimensional counterpart of \eqref{nls},
$$
iu_t+\Delta u-\Delta^2u+|u|^2u=0,
$$
and the {\it biharmonic} cubic nonlinear Schr\"odinger equation
$$
iu_t+\epsilon\Delta^2u+|u|^2u=0, \quad \epsilon=\pm1,
$$
has attracted the attention of many researchers in the past few years (see \cite{bks}, \cite{cg}, \cite{fip}, \cite{mxz}, \cite{pa}, \cite{pa1} \cite{pax}, \cite{zyz1}, \cite{zyz2},\cite{zyz3} and references therein). In most of the mentioned manuscripts, the authors are concerned with local and global well-posedness and formation of singularities as well.  In particular, in various scenarios, such equations present similar dynamics as the standard cubic nonlinear Schr\"odinger equation.

In this paper, we are particularly interested in the existence and nonlinear stability of standing-wave solutions for \eqref{nls}. To the best of our knowledge, this issue has not been addressed in the current literature and such a study could lead us to a better understanding of the dynamics associated with \eqref{nls}. Standing waves are finite-energy waveguide solutions of \eqref{nls} having the form
\begin{equation}\label{standing}
u(x,t)=e^{i\alpha t}\phi(x),
\end{equation}
where $\alpha$ is a real constant and $\phi:\R\to\R$ is a smooth function  satisfying $\phi(x)\to0$, as $|x|\to+\infty$.
 From the numerical point of view, existence of standing-wave solutions and stability with respect to both small perturbations and finite disturbances, for the generalized equation
\begin{equation}\label{kareq}
iu_t+u_{xx}-u_{xxxx}+|u|^{2p}u=0
\end{equation}
were addressed, for instance, in \cite{kar2} and \cite{kar3}. The numerical investigation shows the existence of standing waves for any integer $p\geq1$. In addition, such solutions are stable if $1\leq p\leq2$ and unstable otherwise.

Here, we specialize \eqref{kareq} in the case $p=1$ and exhibit an explicit solution for a suitable value of the parameter $\alpha$. Indeed, by substituting \eqref{standing} into \eqref{nls}, we obtain the fourth-order nonlinear ODE
\begin{equation}\label{soleq}
\phi''''-\phi''+\alpha \phi-\phi^3=0.
\end{equation}
The ansatz $\phi(x)=a\sech^2(bx)$ produces, for
\begin{equation}\label{velocity}
\alpha=\frac{4}{25},
\end{equation}
the solution (see also \cite{kar2} and \cite{waz})
\begin{equation}\label{solitary}
\phi(x)=\sqrt{\frac{3}{10}}\sech^2\left(\sqrt{\frac{1}{20}}\,x\right).
\end{equation}
Hence our main interest in the present paper is to show that the standing
wave \eqref{standing} with $\phi$ given in \eqref{solitary} is
orbitally stable in the energy space, complementing the results in
\cite{kar2}.

Roughly speaking, we say that the standing wave $\phi$
is \textit{orbitally stable}, if the profile of an initial data
$u_0$ for (\ref{nls}) is close to $\phi$, then the associated
evolution in time $u(t)$, with $u(0)=u_0$, remains close to $\phi$,  up to symmetries, for
all values of $t$ (see Definition \ref{stadef} for the precise definition). As pointed out in \cite{kar2}, the problem of
stability is close related with the stabilization of the
self-focusing and the collapse by high order dispersion. This brings
fundamental importance to the nonlinear wave dynamics.

The strategy to prove our stability result is based on the construction of a suitable Lyapunov functional. In fact, we follow the leading arguments in \cite{st}, where the author established the orbital stability of standing waves for abstract Hamiltonian systems of the form
\begin{equation}\label{stha}
Ju_t(t)=H'(u(t))
\end{equation}
posed on a Hilbert space $X$, where $J$ is an invertible bounded operator in $X$. In particular, it is assumed in \cite{st}  that \eqref{stha} is invariant under the action of a one-dimensional group. This enabled the author to prove the orbital stability of standing waves for a large class of nonlinear Schr\"odinger-type equation with potential. It should be noted, however, that the general theory presented in \cite{st} cannot be directly  applied to our case because \eqref{nls} is invariant under phase and translation symmetries. This implies that \eqref{nls} is invariant under the action of a two-dimensional group. Hence, in the present,  we modify the approach in \cite{st} in order to enclose our problem.

It is well-understood by now that the general stability theory
developed in \cite{gss1} and \cite{gss2} is a powerful tool to prove
the orbital stability of standing-wave solutions for abstract Hamiltonian systems. One of the
assumptions in such a theory is that the underlying standing wave
belongs to a smooth curve of standing waves, $\alpha\in
I\mapsto\phi_{\alpha}$, depending on the phase
parameter $\alpha$. In the meantime, \eqref{soleq} does not possesses another
solution of the form \eqref{solitary} for $\alpha\neq4/25$.  This
prevents us in using the stability theory in \cite{gss1}, \cite{gss2} as well as the classical theories as in \cite{bss},
\cite{bo}, and \cite{we}. The method we use below does not require this type of information.

As is well-known, most of the general results in the stability theory rely on a deep spectral analysis for the linear operator arising in the linearization of the equation around the standing wave. In our context, such an operator turns out to be a matrix operator containing fourth-order Schr\"odinger-type operators in the principal diagonal. To obtain the spectral properties we need, we make use of the {\it total positivity theory} developed in \cite{al}, \cite{albo}, and \cite{ka}.

At last, we point out that, from the mathematical point of view, orbital stability and instability of standing and traveling waves for fourth-order equations were studied in \cite{cjy} and \cite{le}, where the authors studied, respectively,
$$
iu_t+\Delta^2 u+V(x)u+|u|^{2\sigma}u=0
$$
and
$$
u_{tt}+\Delta^2u+u+f(u)=0,
$$
under suitable assumptions on the the potential $V$ and the nonlinearity $f$. In particular the classical variational approach introduced by Cazenave and Lions was used. It should be noted however that, since uniqueness is not known, the variational approach provides the orbital stability of the set of minimizers  and not the stability of the standing wave itself.

The paper is organized as follows. In Section \ref{not} we introduce some notation and give the preliminaries result. In particular, we review the total positivity theory for the study of the spectrum of linear operators given as pseudo-differential operators. Such a theory is fundamental to establish the spectral properties given in Section \ref{sepecsec}. Section \ref{stasec} is devoted to show our main results, where we construct a suitable Lyapunov function and prove our stability result for the standing wave \eqref{standing}.

\section{Notation and Preliminaries} \label{not}

Let us introduce some notation used throughout the paper. Given $s\in\R$, by $H^s:=H^s(\R)$ we denote the usual Sobolev space of real-valued functions. In particular $H^0(\R)\simeq L^2(\R)$. The scalar product in $H^s$ will be denoted by $(\cdot,\cdot)_{H^s}$. We set
$$
\Ll^2=L^2(\R)\times L^2(\R)\qquad
\mbox{and}\qquad
\Hh^s=H^s(\R)\times H^s(\R).
$$
Such spaces are endowed with their usual norms and scalar products.

It is not difficult to see that equation \eqref{nls} conserves the energy
\begin{equation}\label{hamiltocons}
E(u)=\frac{1}{2}\int_\R (|u_{xx}|^2+|u_x|^2-\frac{1}{2}|u|^4)\,dx,
\end{equation}
and the mass
\begin{equation}\label{mass}
F(u)=\frac{1}{2}\int_\R |u|^2\,dx.
\end{equation}

Equation \eqref{nls} can also be viewed   as a real Hamiltonian system. In fact,
by writing $u=P+iQ$ and separating real and imaginary parts, we see that
\eqref{nls} is equivalent to the system
\begin{equation}\label{equisys}
\left\{\begin{array}{cccc}
P_t+Q_{xx}-Q_{xxxx}+Q(P^2+Q^2)=0,\\
-Q_t+P_{xx}-P_{xxxx}+P(P^2+Q^2)=0
\end{array}\right.
\end{equation}
In addition, the quantities \eqref{hamiltocons} and \eqref{mass} become
\begin{equation}\label{hamiltocons1}
E(P,Q)=\frac{1}{2}\int_\R \bigg(P_{xx}^2+Q_{xx}^2+P_x^2+Q_x^2-\frac{1}{2}(P^2+Q^2)^2\bigg)\,dx,
\end{equation}
\begin{equation}\label{mass1}
F(P,Q)=\frac{1}{2}\int_\R (P^2+Q^2)\,dx.
\end{equation}
As a consequence, \eqref{equisys} or, equivalently, \eqref{nls} can be written as
\begin{equation}\label{hamiltonian}
\frac{d}{dt}U(t)=J E'(U(t)), \qquad U={P\choose Q},
\end{equation}
where $E'$ represents the Fr\'echet derivative of $E$ with respect to $U$, and
\begin{equation}\label{J}
J=\left(\begin{array}{cccc}
 0 & -1\\
1 & 0
\end{array}\right).
\end{equation}
It is easily seen that
\begin{equation}\label{A}
J^{-1}=-J.
\end{equation}

Note that \eqref{nls} is invariant under the unitary action of
rotation and translation, that is, if $u=u(x,t)$ is a solution of
\eqref{nls} so are $e^{-i\theta}u$ and $u(x-r,t)$, for any real numbers $\theta$ and $r$. Equivalently,
this means if $U=(P,Q)$ is a solution of \eqref{hamiltonian}, so are
\begin{equation}\label{l0}
T_1(\theta)U:=\left(\begin{array}{cccc}
 \cos\theta & \sin\theta\\
-\sin\theta & \cos\theta
\end{array}\right)
\left(\begin{array}{c}
 P\\
 Q
\end{array}\right)
\end{equation}
and
\begin{equation}\label{l00}
T_2(r)U:= \left(\begin{array}{c}
 P(\cdot-r,\cdot)\\
  Q(\cdot-r,\cdot)
\end{array}\right).
\end{equation}

The actions $T_1$ and $T_2$ define unitary groups in $\Hh^2$ with
infinitesimal generators given, respectively, by
\begin{equation*}
T_1'(0)U:=\left(\begin{array}{cccc}
 0 & 1\\
-1 & 0
\end{array}\right)
\left(\begin{array}{c}
 P\\
 Q
\end{array}\right)\equiv -J\left(\begin{array}{c}
 P\\
 Q
\end{array}\right)
\end{equation*}
and
\begin{equation*}
T_2'(0)U:=\partial_x \left(\begin{array}{c}
 P\\
  Q
\end{array}\right).
\end{equation*}

In this context, a standing wave solution having the form \eqref{standing} becomes a solution of \eqref{equisys} of the form
\begin{equation}\label{mstanding}
U(x,t)={\phi(x)\cos(\alpha t)\choose \phi(x)\sin(\alpha t)}.
\end{equation}
Thus, the  function $U$ in \eqref{mstanding},  with $\alpha$ and $\phi$ given, respectively, in \eqref{velocity} and \eqref{solitary}, is a standing wave solution of \eqref{equisys}.

It is easy to see, from \eqref{soleq}, that $(\phi,0)$ is a critical point of
the functional $E+\alpha F$, that is,
\begin{equation}\label{crit}
E'(\phi,0)+\alpha F'(\phi,0)=0.
\end{equation}

To simplify the notation, in what follows we set
\begin{equation}\label{Phi}
\Phi:=(\phi,0)
\end{equation}
and
\begin{equation}\label{linear}
G:=E+\alpha F,
\end{equation}
in such a way that \eqref{crit} becomes $G'(\Phi)=0$.

Let us now introduce the linear operator
\begin{equation}     \label{L}
\mathcal{L}:= \left(\begin{array}{cccc}
\mathcal{L}_1 & 0\\\\
0 & \mathcal{L}_2
\end{array}\right),
\end{equation}
where
\begin{equation}\label{L1}
\mathcal{L}_1:= \partial_x^4-\partial_x^2+\alpha-3\phi^2
\end{equation}
and
\begin{equation}\label{L2}
\mathcal{L}_2:= \partial_x^4-\partial_x^2+\alpha-\phi^2.
\end{equation}
This operator appears in the linearization of \eqref{equisys} around the wave $\Phi$. The knowledge of its spectrum is cornerstone in the analysis to follow.

\subsection{Review of the total positivity theory}\label{reviewsec}

To study the spectrum of the above mentioned operators, we use the
total positivity theory established in \cite{al}, \cite{albo}, and
\cite{ka}. To begin with, let us recall the framework put forward in
\cite{al}, \cite{albo}. Let $\mathcal{T}$ be the operator defined on
a dense subspace of $L^2(\R)$ by
\begin{equation}\label{opT}
\mathcal{T}g(x)=Mg(x)+ \omega g(x)-\varphi^p(x)g(x),
\end{equation}
where $p\geq1$ is an integer, $\omega>0$ is a real parameter, $\varphi$ is real-valued solution of
$$
(M+\omega)\varphi=\frac{1}{p+1}\varphi^{p+1}.
$$
having a suitable decay at infinity, and $M$ is defined as a Fourier multiplier operator by
$$
\widehat{Mg}(\xi)=m(\xi)\widehat{g}(\xi).
$$
Here circumflexes denotes the Fourier transform, $m(\xi)$ is a measurable,
locally bounded, even function on $\R$ satisfying
\begin{itemize}
    \item[(i)] $A_1|\xi|^\mu\leq m(\xi)\leq A_2|\xi|^\mu$ for $|\xi|\geq
    \xi_0$;
    \item[(ii)] $m(\xi)\geq 0$;
\end{itemize}
where $A_1$, $A_2$, and $\xi_0$ are positive real constants, and $\mu\geq1$. Under the above assumptions we
have the following.

\begin{lemma}\label{gelema}
The operator $\mathcal{T}$ is a closed, unbounded, self-adjoint
operator on $L^2(\R)$ whose spectrum consists of the interval
$[\omega,\infty)$ together with a finite number of discrete
eigenvalues in the interval $(-\infty,\omega]$, in which all of them
have finite  multiplicity. In addition, zero is an
eigenvalue of $\mathcal{T}$ with eigenfunction  $\varphi'$.
\end{lemma}
\begin{proof}
See Proposition 2.1 in \cite{al}.
\end{proof}

In order to obtain additional spectral properties of $\mathcal{T}$
let us introduce the family of operators $\mathcal{S}_\theta$, $\theta\geq0$,
on $L^2(\R)$ by
$$
\mathcal{S}_{\theta} g(x)=
\frac{1}{w_\theta(x)}\int_{\R}K(x-y)g(y)dy,
$$
where $K(x)=\widehat{\varphi^p}(x)$ and $w_\theta(x)=m(x)+\theta+\omega$.
These operators act on the Hilbert space
\begin{equation}\label{spaceX}
X=\left\{g\in L^2(\R);\,
\|g\|_{X,\theta}=\left(\int_{\R}|g(x)|^2w_\theta(x)dx\right)^{1/2}<\infty\right\}.
\end{equation}

Since $\mathcal{S}_{\theta}$ is a compact, self-adjoint operator on
$X$ with respect to the norm $\|\cdot\|_{X,\theta}$ (see Proposition
2.2 in \cite{al}), it has a family of eigenfunctions
$\{\psi_{i,\theta}(x)\}_{i=0}^\infty$ forming  an orthonormal basis
of $X$. Moreover, the corresponding eigenvalues
$\{\lambda_i(\theta)\}_{i=0}^\infty$ are real and can be numbered in
order of decreasing absolute value:
$$|\lambda_0(\theta)|\geq|\lambda_1(\theta)|\geq\ldots\geq0.$$

Let us recall some results of \cite{al} and \cite{albo}. The first
one is concerned with an equivalent formulation of the eigenvalue
problem associated with $\mathcal{T}$.

\begin{lemma}\label{multilema}
Suppose $\theta\geq0$. Then $-\theta$ is an eigenvalue of
$\mathcal{T}$ (as an operator on $L^2(\R)$) with eigenfunction $g$
if, and only if, $1$ is an eigenvalue of $\mathcal{S}_{\theta}$ (as
an operator on $X$) with eigenfunction $\widehat{g}$. In particular,
both eigenvalues have the same multiplicity.
\end{lemma}
\begin{proof}
See Corollary 2.3 in \cite{al}.
\end{proof}

The second result is a Krein-Rutman-type theorem.

\begin{lemma}\label{simplelema}
The eigenvalue $\lambda_0(0)$ of $\mathcal{S}_{0}$ is positive,
simple, and has a strictly positive eigenfunction $\psi_{0,0}(x)$.
Moreover, $\lambda_0(0)>|\lambda_1(0)|$.
\end{lemma}
\begin{proof}
See Lemma 8 in \cite{albo}.
\end{proof}

Recall that a function $h:\R\to\R$ is said to be in the class PF(2)
if:
\begin{itemize}
\item[(i)] $h(x)>0$ for all $x\in\R$;
\item[(ii)] for any $x_1,x_2,y_1,y_2\in\R$ with $x_1<x_2$ and $y_1<y_2$, there holds
$$h(x_1-y_1)h(x_2-y_2)-h(x_1-y_2)h(x_2-y_1)\geq0;$$
\item[(iii)] strict inequality holds in (ii) whenever the intervals $(x_1,x_2)$ and $(y_1,y_2)$ intersect.
\end{itemize}

A sufficient condition for a function to be in class PF(2) is that it is logarithmically concave. More precisely, we have.

\begin{lemma}\label{logcon}
A twice differentiable  positive function $h:\R\to\R$ that satisfies
$(\log h(x))''<0$, $x\neq0$, belongs to the class $ PF(2)$.
\end{lemma}
\begin{proof}
See Lemma 4.3 in \cite{al} or Lemma 10 in \cite{albo}.
\end{proof}

The main theorem in \cite{al} reads as follows.

\begin{theorem}\label{mainal}
Suppose $\widehat{\varphi}>0$ on $\R$ and $\widehat{\varphi^p}=:K\in
PF(2)$. Then $\mathcal{T}$ satisfies the following.
\begin{itemize}
    \item[(P1)] $\mathcal{T}$ has a simple, negative eigenvalue $\kappa$;
    \item[(P2)] $\mathcal{T}$ has no negative eigenvalue other than $\kappa$;
    \item[(P3)] the eigenvalue $0$ of $\mathcal{T}$ is simple.
\end{itemize}
\end{theorem}
\begin{proof}
See Theorem 3.2 in \cite{al}.
\end{proof}

\subsection{Local and global well-posedness}
Before we leave this section let us recall the local and global well-posedness results for \eqref{equisys}, or equivalently, for \eqref{nls}.

\begin{theorem}\label{localteo}
Given $U_0\in \Hh^2$, there exists $t_0 > 0$ and a unique solution
$U\in C([0,t_0];\Hh^2)$ of \eqref{equisys} such that $U(0) = U_0$.
The solution has conserved mass and energy in the sense that
$$
F(U(t)) = F(U_0)\quad \mbox{and} \quad E(U(t)) = E(U_0) \quad t\in
[0, t_0],
$$
where  $F$ is defined in \eqref{mass1} and $E$ is defined in
\eqref{hamiltocons1}. If $t^*$ is the maximal time of existence of
$U$, then either
\begin{itemize}
    \item[(i)] $t^*=\infty$, or
    \item[(ii)] $t^*<\infty$ and $\lim_{t\to t^*}\|u(t)\|_{\Hh^2}=\infty$.
\end{itemize}
Moreover, for any $s<t^*$ the map data-solution $U_0\mapsto U$ is
continuous from $\Hh^2$ to $C([0,s];\Hh^2)$.
\end{theorem}
\begin{proof}
See Proposition 4.1 in \cite{pa}.
\end{proof}

The conservation of the energy and the Gagliardo-Nirenberg inequality imply.

\begin{corollary}\label{globalteo}
Given $U_0\in \Hh^2$, the solution $u$ obtained in Theorem \ref{localteo} can
be extended to the whole real line.
\end{corollary}
\begin{proof}
See Corollary 4.1 in \cite{pa}.
\end{proof}

\section{Spectral Analysis} \label{sepecsec}

In this section we use the total positivity theory to get the spectral properties we need to follow. Since $\mathcal{L}$ is a diagonal operator, its eigenvalues are given by the eigenvalues of the operators $\Lum$ and $\Ldois$. Thus, roughly speaking, it suffices to know the spectrum of $\Lum$ and $\Ldois$.

\subsection{The spectrum of $\mathcal{L}_1$}
Attention will be turned to the spectrum of the operator $\Lum$.
Let $\phi$ be the standing wave given in \eqref{solitary} and define
$$
\varphi=\sqrt3\phi.
$$
The operator $\Lum$ then reads as
\begin{equation}\label{L1m}
\mathcal{L}_1= \partial_x^4-\partial_x^2+\alpha-\varphi^2.
\end{equation}

We now can prove the following.

\begin{proposition}\label{specL1}
The operator $\Lum$ in \eqref{L1m} defined on $L^2(\R)$ with domain $H^4(\R)$
has a unique negative eigenvalue, which is simple with positive associated eigenfunction. The eigenvalue zero is
simple with associated eigenfunction $\phi'$. Moreover the rest of the
spectrum is bounded away from zero and the essential spectrum is the interval
$[\alpha,\infty)$.
\end{proposition}
\begin{proof}
First of all  note that, from \eqref{soleq}, $\varphi'$ is an eigenfunction of  $\Lum$ associated with the eigenvalue zero. Also, from Lemma \ref{gelema}, the essential spectrum is exactly $[\alpha,\infty)$. Now observe that $\Lum$ is an operator of the form \eqref{opT} with $m(\xi)=\xi^4+\xi^2\geq0$. It is easy to see that
$$
|\xi|^4\leq m(\xi)\leq 2|\xi|^4, \qquad |\xi|\geq1.
$$
Hence, the operator $\Lum$ satisfies the assumptions in Subsection
\ref{reviewsec} with $A_1=1$, $A_2=2$, $\mu=4$, and $\xi_0=1$. In view
of Lemma \ref{gelema} and Theorem \ref{mainal} it suffices to prove that
$\widehat{\varphi}>0$ and $\widehat{\varphi^2}\in PF(2)$. Since
$$
(\sech^2(\cdot))^\wedge(\xi)=\frac{\pi \xi}{\sinh(\frac{\pi
\xi}{2})}=:\rho(\xi)>0,
$$
and $\varphi=\sqrt3\phi$, it is clear that $\widehat{\varphi}>0$. On the
other hand, by the properties of the Fourier transform, $\widehat{\varphi^2}$ is the convolution of a function similar to
$\rho$ with itself. Since $\rho$ is logarithmically concave (see proof of Theorem 4.6 in \cite{al}) and the
convolution of logarithmically concave functions also is logarithmically
concave (\cite[Section 4]{bl}), we obtain that $\widehat{\varphi^2}\in PF(2)$. The fact that the eigenfunction associated with the negative eigenvalue can be taken to be positive follows as in \cite[Proposition 2]{abh}.
\end{proof}

\subsection{The spectrum of $\Ldois$}

As before, let $\phi$ be the solitary wave given in
\eqref{solitary}. Observe that $\Ldois$ is of the form \eqref{opT},
but $\phi'$ is not an eigenfunction at all. Thus the theory
developed in the last paragraphs cannot be applied directly to
$\Ldois$. On the other hand, to $\Ldois$ we can still associate  a
family of operators $\mathcal{S}_{\theta}$, $\theta\geq0$, as in Subsection \ref{reviewsec}.  A simple
inspection in the proofs of \cite{al} reveals that Lemmas
\ref{multilema} and \ref{simplelema} above remain true here.

As a consequence, concerning the operator $\Ldois$ we have the following spectral properties.

\begin{proposition}\label{specL2}
The operator $\Ldois$ in \eqref{L2} defined on $L^2(\R)$ with domain
$H^4(\R)$ has no negative eigenvalue. The eigenvalue zero is simple with
associated eigenfunction $\phi$. Moreover the rest of the spectrum is bounded
away from zero and the essential spectrum is the interval $[\alpha,\infty)$.
\end{proposition}
\begin{proof}
Although this result was not stated, it is  reminiscent of the theory in
\cite{al}. The fact that the essential spectrum is $[\alpha,\infty)$ can be proved as in \cite[Proposition 1]{abh}. In view of \eqref{soleq} it is clear that zero is an eigenvalue with
eigenfunction $\phi$. Thus, Lemma \ref{multilema} implies that $1$ is an
eigenvalue of $\mathcal{S}_{0}$ with eigenfunction $\widehat{\phi}$.\\

\noindent {\bf Claim.} If  $\lambda_0(0)$ is the first eigenvalue of $\mathcal{S}_{0}$, we have  $\lambda_0(0)=1$.\\

Indeed, assume by contradiction that $\lambda_0(0)\neq1$ and let $\psi_{0,0}$
be the associated eigenfunction. By Lemma \ref{simplelema} we have
$\psi_{0,0}>0$. Since $\widehat{\phi}>0$ (see proof of Proposition \ref{specL1}), the scalar product in $L^2(\R)$ between $\psi_{0,0}$ and
$\widehat{\phi}$ is then strictly positive, which contradicts the fact that
the eigenfunctions are orthogonal.\\

The above claim together with Lemma \ref{simplelema} imply that 1 is
a simple eigenvalue of $\mathcal{S}_{0}$. Therefore, it follows from
Lemma \ref{multilema} that zero is a simple eigenvalue of $\Ldois$.

It remains to show that $\Ldois$ has no negative eigenvalues. To do
so, it is sufficient, from Lemma \ref{multilema}, to prove that 1 is
not an eigenvalue of $\mathcal{S}_{\theta}$ for any $\theta>0$. We
already know that
$$
\lim_{\theta\to\infty}\lambda_0(\theta)=0
$$
and $\theta\in[0,\infty)\mapsto\lambda_0(\theta)$ is a strictly decreasing
function (see \cite[page 9]{al}). Thus, for $\theta>0$ and $i\geq1$,
$$
|\lambda_i(\theta)|\leq \lambda_0(\theta)<\lambda_0(0)=1.
$$
 This obviously implies that 1 cannot be an eigenvalue of $\mathcal{S}_{\theta}$,
 $\theta>0$, and the proof of the proposition is completed.
\end{proof}

\subsection{The spectrum of $\mathcal{L}$}

We finish this section by stating the spectral properties of the ``linearized'' operator $\mathcal{L}$. Indeed, a combination of Propositions \ref{specL1} and \ref{specL2} gives us the following.

\begin{theorem}
The operator $\mathcal{L}$ in \eqref{L} defined on $L^2(\R)\times L^2(\R)$
with domain $H^4(\R)\times H^4(\R)$ has a unique negative eigenvalue, which
is simple. The eigenvalue zero is double with associated eigenfunctions
$(\phi',0)$ and $(0,\phi)$. Moreover the essential spectrum is the interval
$[\alpha,\infty)$.
\end{theorem}

\section{Orbital stability}\label{stasec}

Let us start this section by making clear our notion of orbital stability.  Taking into account that \eqref{hamiltonian} is invariant by the transformations \eqref{l0} and \eqref{l00}, we define the orbit generated by $\Phi$ (see \eqref{Phi}) as
\begin{equation}\label{l1}
\begin{split}
\Omega_\Phi&=\{T_1(\theta) T_2(r)\Phi;\;\;\theta,r\in\R\}\\
&\simeq \left\{ \left(\begin{array}{cc}
 \cos\theta & \sin\theta\\
-\sin\theta & \cos\theta
\end{array}\right)\left(\begin{array}{c}
 \phi(\cdot-r)\\
 0
\end{array}\right);\;\;\theta,r\in\R  \right\}.
\end{split}
\end{equation}
In $\Hh^2$, we introduce the pseudo-metric $d$ by
$$
d(f,g):=\inf\{\|f-T_1(\theta)
T_2(r)g\|_{\Hh^2};\;\theta,r\in\R\}.
$$
Note  that, by definition, the distance between $f$ and $g$ is the distance between $f$ and the orbit generated by  $g$ under the action of rotations and translations. In particular,
\begin{equation}\label{l2}
d(f,\Phi)=d(f,\Omega_\Phi).
\end{equation}

\begin{definition}\label{stadef}
Let $\Theta(x,t)=(\phi(x)\cos(\alpha t), \phi(x)\sin(\alpha t))$ be a standing wave for \eqref{hamiltonian}. We say that $\Theta$ is orbitally stable in $\Hh^2$ provided that, given $\ve>0$, there exists $\delta>0$ with the following property: if $U_0\in \Hh^2$ satisfies $\|U_0-\Phi\|_{\Hh^2}<\delta$, then the solution, $U(t)$, of \eqref{hamiltonian} with initial condition $U_0$ exist for all $t\geq0$ and satisfies
$$
d(U(t),\Omega_\Phi)<\ve, \qquad \mbox{for all}\,\, t\geq0.
$$
Otherwise, we say that $\Theta$ is orbitally unstable in $\Hh^2$.
\end{definition}

\subsection{Positivity of the operator $G''(\Phi)$}

Before proving our main result we need some positivity properties of the operators appearing in \eqref{L}-\eqref{L2}. To do so, we utilize the spectral properties established in Section \ref{sepecsec}.

\begin{lemma}\label{bouL2}
There exists $\delta_2>0$ such that
\begin{equation}\label{a1}
(\Ldois Q,Q)_{L^2}\geq\delta_2\|Q\|_{L^2}^2,
\end{equation}
for all $Q\in H^4$ satisfying $(Q,\phi)_{L^2}=0$.
\end{lemma}
\begin{proof}
Write $L^2=[\phi]\oplus M$ where $\phi\perp Q$ for all $Q\in M$. Since $\phi$ belongs to the kernel of
$\Ldois$ it follows from Theorem 6.17 in \cite{kato} that the spectrum of the
part $\Ldois{\mid}_M$ coincides with $\sigma(\Ldois)\setminus\{0\}$.
Proposition \ref{specL2} and the arguments in \cite[page 278]{kato} imply
that there exists a $\delta_2>0$ such that $\Ldois\geq\delta_2$ on $M\cap
H^4$. Now, if $Q\in H^4$ satisfies $(Q,\phi)_{L^2}=0$ we have $Q\in M$ and
the conclusion of the lemma then follows.
\end{proof}

\begin{lemma}\label{bouL1}
Let $\gamma$ be defined as
$$
\gamma=\inf\{(\Lum P,P)_{L^2}; \;P\in H^4, \|P\|_{L^2}=1,(P,\phi)_{L^2}=0\}.
$$
Then $\gamma=0$.
\end{lemma}
\begin{proof}
By defining
\begin{equation}\label{defa0}
a_0:=\inf\{\langle \Lum P,P\rangle; \;P\in H^2, \|P\|_{L^2}=1,(P,\phi)_{L^2}=0\}
\end{equation}
with
$$
\langle \Lum P,P\rangle=\int_\R\big(P_{xx}^2+P_{x}^2+\alpha P^2-3\phi^2P^2\big)dx,
$$
it suffices to prove that $a_0=0$ and that the infimum in \eqref{defa0} is achieved, because in this case, the Lagrange multiplier theory implies that $\gamma=a_0$.

First of all, note from Proposition \ref{specL1} that $\Lum$ is bounded from below; thus, $a_0$ is finite. Moreover, since $\Lum\phi'=0$ and $(\phi,\phi')_{L^2}=0$, we deduce that $a_0\leq0$. We shall prove that $a_0=0$, by showing  that $a_0\geq0$. To accomplish this, we recall  the following result due to Weinstein \cite{we} (see also \cite[Chapter 6]{an}).

\begin{lemma}\label{welemma}
Let $A$ be a self-adjoint operator on $L^2(\R)$ satisfying:
\begin{itemize}
  \item[(i)] $A$ has exactly one negative simple eigenvalue with positive associated eigenfunction, say, $\varphi$;
  \item[(ii)] zero is an isolated eigenvalue;
  \item[(iii)] there is $\psi\in [Ker(A)]^\perp$ such that $(\psi,\varphi)_{L^2}\neq0$ and
  $$
  -\infty<a_0:=\min\{\langle Af,f\rangle;\;\|f\|_{L^2}=1, (f,\psi)_{L^2}=0\},
  $$
  where $\langle A\cdot,\cdot\rangle$ denotes the quadratic form associated with $A$.
\end{itemize}
 If $I:=(A^{-1}\psi,\psi)_{L^2}\leq0$ then $a_0\geq0$.
\end{lemma}

Now we turn back to the proof of Lemma \ref{bouL1} and apply Lemma \ref{welemma} with $A=\Lum$ and $\psi=\phi$. First, it should be noted that it is implicitly assumed in Lemma \ref{welemma} that $a_0$ is achieved for a suitable function.
 Thus, in what follows, we prove that the infimum in \eqref{defa0} is achieved. This is well-known by now, but for the sake of completeness we bring some details (see also \cite[Appendix]{albo} or \cite[Chapter 6]{an}).

Let $\{P_j\}_{j\in\mathbb{N}}\subset H^2$ be a minimizing sequence, that is, a sequence satisfying $\|P_j\|_{L^2}=1$, $(P_j,\phi)_{L^2}=0$, for all $j\in\mathbb{N}$, and
\begin{equation}\label{bound1}
\la\mathcal{L}_1P_j,P_j\ra\rightarrow a_0,\ \ \ \mbox{as}\ j\rightarrow +\infty.
\end{equation}
Since $\|P_j\|_{L^2}=1$, it is easily seen that
\begin{equation}\label{bound1.1}
0<\alpha\leq \|P_{j,xx}\|^2_{L^2}+\|P_{j,x}\|^2_{L^2}+\alpha\|P_{j}\|^2_{L^2}= \la\mathcal{L}_1P_j,P_j\ra+3\int_{\R}\phi^2P_j^2dx.
\end{equation}
Because $\phi$ is bounded, we see at once that the right-hand side of \eqref{bound1.1} is bounded. Hence,  there is a constant $C>0$ such that $\|P_j\|_{H^2}\leq C$, for all $j\in\mathbb{N}$. So, there are a subsequence, which we still denote by $\{P_j\}$, and  $P\in H^2$ such that $P_j\rightharpoonup P$ in $H^2$. On account of the weak convergence in $L^2$, we have $(P,\phi)_{L^2}=0$.

Taking into account the compactness of the embedding $H^2(-R,R)\hookrightarrow L^2(-R,R)$, for all $R>0$, a standard Cantor diagonalization argument allows us to obtain a subsequence, which we still denote by $\{P_j\}$, such that $P_j\to P$ in $L^2_{loc}(\R)$. By writing, for a fixed large $R>0$,
\begin{equation}\label{bound1.2}
\int_{\R}\phi^2(P_j^2-P^2)dx=\int_{|x|\leq R}\phi^2(P_j^2-P^2)dx+\int_{|x|> R}\phi^2(P_j^2-P^2)dx,
\end{equation}
we see that first term in the right-hand side of \eqref{bound1.2} goes to zero, as $j\to\infty$, because $\phi$ is bounded and $P_j\to P$ in $L^2_{loc}(\R)$. Also, taking the advantage that $\phi(x)\to0$, as $|x|\to\infty$, the second term in the right-hand side of \eqref{bound1.2} can be made sufficiently small if we choose $R$ sufficiently large. As a consequence, we deduce that
\begin{equation}\label{bound2}
\int_{\mathbb{R}}\phi^2P_j^2dx\rightarrow \int_{\mathbb{R}}\phi^2P^2dx,\ \ \ \ \mbox{as}\ j\rightarrow+\infty.
\end{equation}

Now, we claim that $P\neq0$. On the contrary, suppose that $P=0$. Then, \eqref{bound2} implies that $\int_{\mathbb{R}}\phi^2P_j^2dx\rightarrow0$, as $j\to\infty$. By taking the limit in \eqref{bound1.1}, as $j\to\infty$, we obtain, in view of \eqref{bound1},
$$
0\leq\alpha\leq a_0.
$$
This fact generates a contradiction and therefore $P\neq 0$.

Next, \eqref{bound1}, \eqref{bound2}, and the lower semicontinuity of the weak convergence yield
 $$
\la\mathcal{L}_1P,P\ra=a_0\qquad \mbox{and} \qquad \|P\|_{L^2}\leq1.
$$
We now prove that the infimum in \eqref{defa0} is achieved. Indeed, by defining $Q=\frac{P}{\|P\|_{L^2}}$, we see that $\|Q\|_{L^2}=1$ and $(Q,\phi)_{L^2}=0$. Moreover,
\begin{equation}\label{bound4}
a_0 \|P\|_{L^2}^2\leq\|P\|_{L^2}^2\la\mathcal{L}_1Q,Q\ra=\la\mathcal{L}_1P,P\ra=a_0.
\end{equation}
Since $a_0\leq0$, we consider two cases.
\vskip.2cm

\noindent {\bf Case 1.} $a_0<0.$ Here, \eqref{bound4} immediately gives $\|P\|_{L^2}\geq1$. Therefore, $\|P\|_{L^2}=1$ and the infimum is achieved at the function $P$.

\vskip.2cm

\noindent {\bf Case 2.} $a_0=0.$ In this case, since $\la\mathcal{L}_1P,P\ra=\la\mathcal{L}_1Q,Q\ra=0$ and $\|Q\|_{L^2}=1$, the infimum is achieved at the function $Q$.

Finally, from arguments due to Albert \cite[page 17]{al} (see also Remark \ref{remmm} below), we deduce
that the quantity $I$ defined in Lemma \ref{welemma} is negative.  An application of 
Lemma \ref{welemma} then gives us $a_0\geq0$. The proof of the lemma is
thus completed.
\end{proof}

\begin{lemma}\label{bouL11}
Let
\begin{equation}\label{a3}
\delta_1:=\inf\{(\Lum P,P)_{L^2}; \;P\in H^4, \|P\|_{L^2}=1,(P,\phi)_{L^2}=(P,\phi')_{L^2}=0\}.
\end{equation}
Then $\delta_1>0$.
\end{lemma}
\begin{proof}
Since $\{(\mathcal{L}_{1}P,P)_{L^2}; \ \|P\|=1,\ (P, \phi)_{L^2}=0,\
(P, \phi')_{L^2}=0\}\subset\{(\mathcal{L}_{1}P,P)_{L^2}; \ \|P\|=1,\
(P, \phi)_{L^2}=0,\}$, we get from Lemma \ref{bouL1} that $\delta_1\geq0$. In addition, by using similar arguments as in Lemma \ref{bouL1}, we get that the infimum is achieved at a function $\kappa$ . Next, assume by contradiction that $\delta_1=0$.  From
Lagrange's Multiplier Theorem there are $m,n,r\in\mathbb{R}$ such
that
\begin{equation}\label{lagrange}
\mathcal{L}_{1}\kappa=m\kappa+n\phi +r\phi'.
\end{equation}
Now, $(\mathcal{L}_{1}\kappa,\kappa)_{L^2}=\delta_1=0$ immediately
implies $m=0$. On the other hand, from the self-adjointness of
$\mathcal{L}_{1}$, we see that
$0=(\mathcal{L}_{1}\phi',\kappa)_{L^2} =r(\phi',\phi')_{L^2}$ and
$r=0$. These facts allow us to deduce that $\mathcal{L}_{1}\kappa=n\phi$. We now consider two cases:

\vskip.2cm

\noindent {\bf Case 1.} $n=0.$ Here, we have $\mathcal{L}_{1}\kappa=0$. Since $\ker(\mathcal{L}_{1})=[\phi']$ we obtain $\kappa=\mu\phi'$ for some $\mu\neq0$. This fact generates a
contradiction because $\kappa\bot\phi'$.

\vskip.2cm

\noindent {\bf Case 2.} $n\neq0.$ In this case, the Fredholm theory implies that the equation $\mathcal{L}_{1}\chi=\phi$ has a unique solution satisfying $(\chi,\phi')_{L^2}=0$. Since $n\neq0$, one sees that such a solution is $\chi=\kappa/n$. This allows us, in view of the definition of $\kappa$, deducing that $I:=(\chi,\phi)_{L^2}=0$. On the other hand, since $I$ does not depend on the choice of $\chi$ satisfying $\mathcal{L}_{1}\chi=\phi$, the arguments in \cite{al} gives us that $I\neq0$ (see Remark \ref{remmm}), which is a contradiction.

The lemma is thus proved.
\end{proof}

\begin{remark}\label{remmm}
In the proofs of Lemmas \ref{bouL1} and \ref{bouL11}, we used that $I:=(\chi,\phi)_{L^2}<0$. The theory established in Albert
 \cite{al} presents the construction of a function $\chi$ satisfying  $\mathcal{L}_{1}\chi=\phi$.  Actually, Albert has considered more general differential operators having the form
 $$
 \mathcal{L}=M+\alpha-\phi^p,
 $$
 where $M$ is a differential operator of order $2n$, $\alpha$ is real number, $p$ is an integer, and $\phi$ is a $\frac{2n}{p}$-power of the hyperbolic-secant function. The number $I$ is still defined by $I=(\chi,\phi)_{L^2}$, with $\chi$ satisfying $\mathcal{L}\chi=\phi$. We   briefly discuss the method of how to show that $I<0$, by defining
$$\eta=\sum_{i=0,i\neq1}^{\infty}\left(\frac{1}{1-\lambda_{i}(0)}\right)
\left\langle\frac{\widehat{\phi}}{w_0},\psi_{i,0}\right\rangle_X\psi_{i,0},
$$
where $X$ is the space defined in \eqref{spaceX} and $\langle\cdot,\cdot\rangle_X$ stands for its usual inner product.
Standard arguments enable us to deduce that the series 
converges in $X$, and so $\eta\in X$. Thus, we can choose
$\chi\in L^2(\R)$ such that $\widehat{\chi}=\eta$. A direct
computation gives  $\widehat{\mathcal{L}\chi}=\widehat{\phi}$. By
using the orthogonality of Gegenbauer's polynomials (which give
explicitly all the eigenfunctions $\psi_{i,0}$, $i=0,1,2,\ldots$) one deduces that
$$
I=a\sum_{j=0}^\infty\left(\frac{\lambda_{2j}}{1-\lambda_{2j}}\right)\left\{ \frac{\Gamma(2j+1)\cdot(2j+n+r-\frac{1}{2}}{\Gamma(2j+2n+2r-1)} \right\} \left\{ \frac{\Gamma(j+n)\Gamma(j+n+r-\frac{1}{2}}{\Gamma(j+1)\Gamma(j+r+\frac{1}{2})} \right\}^2,
$$
where $a=\left( \frac{2^{n+r-1}\Gamma(r)}{\pi\Gamma(n)}  \right)^2$, $r=\frac{2n}{p}$, and $\Gamma$ represents the well known Gamma function. For $n=2$ (which is our case), by using numerical computations, Albert showed that $I<0$ as long as $p<4.82$ (approximately). In our case, $p=2$.
\end{remark}

\begin{remark}
The construction presented in Remark \ref{remmm} is useful when we do not have a smooth curve of standing waves, depending on the parameter $\alpha$.
As is well known in the current literature, when dealing with a general nonlinear dispersive equation, the existence of a smooth curve of standing waves, say, $c\in \mathcal{I}\mapsto \phi_c$, $\mathcal{I}\subset \mathbb{R}$, parametrized by the wave velocity $c$ is sufficient to construct  an explicit $\chi$ in terms of $\frac{d\phi_c}{dc}$ (see e.g., \cite{al}, \cite{abh}, \cite{an}, \cite{bss}, \cite{bo}, \cite{gss1}, \cite{gss2}, \cite{st}, \cite{we}, and references therein). As we have mentioned in the introduction, here we do not have a smooth curve of explicit standing waves.
\end{remark} \label{rem4.6}

\begin{corollary}\label{corbou1}
Assume that $v=(P,Q)\in \Hh^4$ is such that
\begin{equation}\label{a4}
(Q,\phi)_{L^2}=(P,\phi)_{L^2}=(P,\phi')_{L^2}=0.
\end{equation}
Then, there is $\delta>0$ such that
\begin{equation}\label{a5}
(\mathcal{L}v,v)_{\Ll^2}\geq \delta\|v\|_{\Ll^2}^2.
\end{equation}
\end{corollary}
\begin{proof}
From Lemmas \ref{bouL2} and \ref{bouL11} there are $\delta_1,\delta_2>0$ such that $(\Ldois Q,Q)_{L^2}\geq\delta_2\|Q\|_{L^2}^2$ and $(\Lum P,P)_{L^2}\geq\delta_1\|P\|_{L^2}^2$. Since $\mathcal{L}v=(\Lum P, \Ldois Q)$, it is sufficient to take $\delta=\min\{\delta_1,\delta_2\}$.
\end{proof}

\begin{lemma}\label{gard}
There exist positive constants $\varepsilon$ and $C$ such that
$$
(\mathcal{L}v,v)_{\Ll^2}\geq \varepsilon\|v\|_{\Hh^2}^2 -C\|v\|_{\Ll^2}^2,
$$
for all $v=(P,Q)\in \Hh^4$.
\end{lemma}
\begin{proof}
From Garding's inequality (see \cite[page 175]{yo}), there exist positive constants $\varepsilon_1,\varepsilon_2,C_1$, and $C_2$
such that
$$
(\Lum P,P)_{L^2}\geq \varepsilon_1\|P\|^2_{H^2}-C_1\|P\|^2_{L^2}, \quad (\Ldois Q,Q)_{L^2}\geq \varepsilon_2\|Q\|^2_{H^2}-C_2\|Q\|^2_{L^2}.
$$
By the definition of $\mathcal{L}$ it suffices to choose $\varepsilon=\min\{\varepsilon_1,\varepsilon_2\}$ and $C=\max\{C_1,C_2\}$.
\end{proof}

\begin{remark}\label{remI}
It is not difficult to see that $\mathcal{L}$ is the unique self-adjoint linear operator such that
\begin{equation}\label{a6}
\langle G''(\Phi)v,z\rangle =(\mathcal{L}v,z)_{\Ll^2}, \qquad v\in \Hh^6, \; z\in \Hh^2,
\end{equation}
where $G''$ represents the second order Fr\'echet derivative of $G$.
In particular, $G''(\Phi)v=\mathcal{I}\mathcal{L}v$, $v\in \Hh^6$, where $\mathcal{I}:\Hh^2\to \Hh^{-2}$ is the natural injection of $\Hh^2$ into $\Hh^{-2}$ with respect to the inner product in $\Ll^2$, that is,
\begin{equation}\label{a6.0}
\langle \mathcal{I}u,v\rangle=(u,v)_{\Ll^2}, \qquad u,v\in \Hh^2.
\end{equation}
See Lemma 3.3 in \cite{st} for details.
\end{remark}

\begin{lemma}\label{lemasc}
Assume that $v=(P,Q)\in \Hh^2$ is such that
\begin{equation*}
(Q,\phi)_{L^2}=(P,\phi)_{L^2}=(P,\phi')_{L^2}=0.
\end{equation*}
Then, there is $\delta>0$ such that
\begin{equation*}
\langle G''(\Phi) v,v\rangle\geq \delta\|v\|_{\Hh^2}^2.
\end{equation*}
\end{lemma}
\begin{proof}
By density it suffices to assume that $v$ belongs to $\Hh^6$. From Corollary \ref{corbou1} and Lemma \ref{gard}, one infers that
$$
\left(1+\dfrac{C}{\delta}\right)  (\mathcal{L}v,v)_{\Ll^2}\geq \varepsilon\|v\|_{\Hh^2}^2,
$$
that is,
$$
 (\mathcal{L}v,v)_{\Ll^2}\geq \frac{\varepsilon\delta}{C+\delta}\|v\|_{\Hh^2}^2.
$$
The conclusion then follows from \eqref{a6}.
\end{proof}

In what follows, let $\RR:\Hh^2\to \Hh^{-2}$ be the Riesz isomorphism with respect to inner product in $\Hh^2$, that is,
\begin{equation}\label{a6.1}
\langle \mathcal{R}u,v\rangle=(u,v)_{\Hh^2}, \qquad u,v\in \Hh^2.
\end{equation}

Lemma \ref{lemasc} establishes the positivity of $G''(\Phi)$ under an orthogonality condition in $L^2$. Next lemma shows that the same positivity holds if we assume the orthogonality in $H^2$.

\begin{lemma}\label{prinlemma}
Let $\mathcal{I}$ be the operator defined  in Remark \ref{remI} and $J$ is as in \eqref{J}. Let
\begin{equation}\label{defZ}
\begin{split}
\Z&=\{J\Phi,\Phi',\RR^{-1}\mathcal{I}\Phi\}^{\perp}\\
&=\{z\in \Hh^2; (z,\Phi)_{\Hh^2}=(z,\Phi')_{\Hh^2}=(z,\RR^{-1}\mathcal{I}\Phi)_{\Hh^2}=0   \}.
\end{split}
\end{equation}
Then, there exists $\delta>0$ such that
\begin{equation}\label{a6.2}
\langle G''(\Phi)z,z\rangle\geq \delta \|z\|^2_{\Hh^2}, \qquad z\in \Z.
\end{equation}
\end{lemma}
\begin{proof}
We start by defining
$$\vf_1=-\dfrac{1}{\|J\Phi\|_{\Ll^2}}J\Phi, \quad \vf_2=\dfrac{1}{\|\Phi'\|_{\Ll^2}}\Phi', \quad \psi_1=J\vf_1, \qquad \psi_2=J\vf_2.
$$
It is clear that, $\|\vf_1\|_{\Ll^2}=\|\vf_2\|_{\Ll^2}=\|\psi_1\|_{\Ll^2}=\|\psi_2\|_{\Ll^2}=1$. In addition, since $J$ is skew-symmetric,
\begin{equation}\label{a7}
(\vf_1,\psi_1)_{\Ll^2}=(\vf_2,\psi_2)_{\Ll^2}=0.
\end{equation}
Also, from the definition, we promptly deduce
\begin{equation}\label{a8}
(\vf_1,\vf_2)_{\Ll^2}=(\psi_1,\vf_2)_{\Ll^2}=0.
\end{equation}
Take any $z\in\Z$ and define
$$
v:=z-(z,\vf_1)_{\Ll^2}\vf_1-(z,\vf_2)_{\Ll^2}\vf_2-(z,\psi_1)_{\Ll^2}\psi_1.
$$
Let us show that $v$ satisfies the assumptions in Lemma \ref{lemasc}. Indeed, from \eqref{a7} and \eqref{a8}, it is not difficult to see that
$$
(v,\vf_1)_{\Ll^2}=(v,\vf_2)_{\Ll^2}=(v,\psi_1)_{\Ll^2}=0.
$$
If we assume that $v$ writes as $(P,Q)$, these last equalities reduces to
$$
(Q,\phi)_{L^2}=(P,\phi)_{L^2}=(P,\phi')_{L^2}=0.
$$
An application of Lemma \ref{lemasc} yields the existence of $\delta>0$ such that
\begin{equation}\label{a9}
\langle G''(\Phi) v,v\rangle\geq \delta\|v\|_{\Hh^2}^2.
\end{equation}
In view of \eqref{a6.1}, \eqref{a9} is equivalent to
\begin{equation}\label{a10}
(\RR^{-1} G''(\Phi) v,v)_{\Hh^2}\geq \delta\|v\|_{\Hh^2}^2.
\end{equation}
Note that, from \eqref{a6.0} and \eqref{a6.1},
\begin{equation}\label{a10.1}
(f,g)_{\Ll^2}=\langle \I f,g\rangle=(\RR^{-1}\I f,g)_{\Hh^2}, \qquad f,g\in\Hh^2.
\end{equation}
Hence, by hypotheses,
$$
(z,\psi_1)_{\Ll^2}=\dfrac{1}{\|J\Phi\|_{\Ll^2}}(\Phi,z)_{\Ll^2}=\dfrac{1}{\|J\Phi\|_{\Ll^2}}(\RR^{-1}\I
\Phi,z)_{\Hh^2}=0,
$$
which implies that
\begin{equation}\label{a11}
v:=z-(z,\vf_1)_{\Ll^2}\vf_1-(z,\vf_2)_{\Ll^2}\vf_2.
\end{equation}
Since $\vf_1$ and $\vf_2$ are smooth functions and $\mathcal{L}\vf_1=\mathcal{L}\vf_2=0$, from \eqref{a6}, one infers that $G''(\Phi)\vf_1=G''(\Phi)\vf_2=0$. Therefore, \eqref{a11} gives
\begin{equation}\label{a12}
G''(\Phi)z=G''(\Phi)v.
\end{equation}
To simplify notation, let us set $\alpha_1=(z,\vf_1)_{\Ll^2}$, $\alpha_2=(z,\vf_2)_{\Ll^2}$, and $S:\Hh^2\to\Hh^2$ the self-adjoint operator defined by $S:=\RR^{-1}G''(\Phi)$. Thus, from the assumption,
$$
\|z\|_{\Hh^2}^2=(z,v+\alpha_1\vf_1+\alpha_2\vf_2)_{\Hh^2}=(z,v)_{\Hh^2}+\alpha_1(z,\vf_1)_{\Hh^2}+\alpha_2(z,\vf_2)_{\Hh^2}=(z,v)_{\Hh^2}.
$$
The Cauchy-Schwartz inequality then implies
$$
\|z\|_{\Hh^2}^2=|(z,v)_{\Hh^2}|\leq \|z\|_{\Hh^2}\|v\|_{\Hh^2},
$$
that is, $\|z\|_{\Hh^2}\leq \|v\|_{\Hh^2}$. Also, since $S$ is self-adjoint, it follows from \eqref{a12} that
$$
(Sz,z)_{\Hh^2}=(Sv,v)_{\Hh^2}+\alpha_1(Sv,\vf_1)_{\Hh^2}+\alpha_2(Sv,\vf_2)_{\Hh^2}=(Sv,v)_{\Hh^2}.
$$
Finally, combining this last equality with \eqref{a10},
$$
(Sz,z)_{\Hh^2}=(Sv,v)_{\Hh^2}\geq \delta\|v\|_{\Hh^2}^2\geq \delta \|z\|_{\Hh^2}^2,
$$
which is our statement \eqref{a6.2} in view of \eqref{a6.1}.
\end{proof}

\subsection{Lyapunov function and stability}\label{lysec}

Before proceeding, let us introduce a notation.  Given any real number $\ve>0$, by $\Omega_\Phi^\ve$, we denote the $\ve$-neighborhood of $\Omega_\Phi$, that is,
 $$
 \Omega_\Phi^\ve=\{v\in \Hh^2;\; d(v,\Omega_\Phi)<\ve\}.
 $$

In what follows, the following lemma will be necessary.

\begin{lemma}\label{mlema}
There exists $R>0$, depending only on $\Phi$, such that for all $\rho\in(0,R)$ and $v\in  \Omega_\Phi^\rho$, there exist $r_1,\theta_1\in\R$ satisfying
\begin{equation}\label{l3}
\|v-T_1(\theta_1) T_2({r_1})\Phi\|_{\Hh^2}<\rho
\end{equation}
and
\begin{equation}\label{l4}
\big(v-T_1(\theta_1) T_2({r_1})\Phi,J T_1(\theta_1)
T_2({r_1})\Phi)\big)_{\Hh^2} = \big(v-T_1(\theta_1)
T_2({r_1})\Phi,T_1(\theta_1) T_2({r_1})\Phi')\big)_{\Hh^2}=0.
\end{equation}
\end{lemma}
\begin{proof}
For all $\theta,r\in\R$, we have
\begin{equation}\label{l5}
\begin{split}
\|\Phi-T_1(\theta) T_2(r)\Phi\|^2_{\Hh^2}& =
\|\phi-\phi(\cdot-r)\cos\theta\|^2_{H^2}+\|\phi(\cdot-r)\sin\theta\|^2_{H^2}\\
&=2\|\phi\|^2_{H^2}-2\cos\theta\big(\phi,\phi(\cdot-r)\big)_{H^2}.
\end{split}
\end{equation}
Using the Lebesgue convergence theorem, it is easy to see that $\big(\phi,\phi(\cdot-r)\big)_{H^2}\to0$, as $|r|\to\infty$. Choose $\widehat{r}>0$ such that  $|\big(\phi,\phi(\cdot-r)\big)_{H^2}|\leq \|\phi\|^2_{H^2}/2$ for all $r\in\R$ with $|r|\geq\widehat{r}$. Thus, in view of \eqref{l5},
\begin{equation}\label{l5.1}
\|\Phi-T_1(\theta) T_2(r)\Phi\|^2_{\Hh^2}\geq
2\|\phi\|^2_{H^2}-\|\phi\|^2_{H^2}=\|\phi\|^2_{H^2},
\end{equation}
for all $r\in\R$ with $|r|\geq\widehat{r}$, uniformly with respect to $\theta\in\R$.

Fix $R=\|\phi\|^2_{H^2}/2$.  If $\rho\in(0,R)$ and $v\in
\Omega_\Phi^\rho$,  then by definition, there exist
$\theta_0,r_0\in\R$ satisfying $\|v-T_1({\theta_0})
T_2({r_0})\Phi\|_{\Hh^2}<\rho$, that is, $T_1({\theta_0})
T_2({r_0})\Phi\in B_\rho(v)$, where $B_\rho(v)$ denotes the open
ball in $\Hh^2$ centered at $v$ with radius $\rho$. Let $\Lambda$ be
the largest connected open set in $\R^2$ containing the point
$(\theta_0,r_0)$ such that $T_1({\theta}) T_2({r})\Phi\in B_\rho(v)$
for all $(\theta,r)\in\Lambda$.

\vskip.2cm
{\bf Claim.} $\Lambda$ is bounded.
\vskip.2cm

Indeed, first note that
$$
\|\Phi-T_1(\pm{\pi/2})T_2(r)\Phi\|^2_{\Hh^2} =
\|\phi\|^2_{H^2}+\|\phi(\cdot-r)\|^2_{H^2}\geq\|\phi\|^2_{H^2},
$$
that is, for all $r\in\R$,
\begin{equation}\label{l6}
\|\Phi-T_1(\pm{\pi/2})T_2(r)\Phi\|^2_{\Hh^2}\geq2R.
\end{equation}
Hence,
\begin{equation*}
\begin{split}
\|v-&T_1({\theta_0\pm\pi/2})T_2(r+r_0)\Phi\|_{\Hh^2}\\
&\geq
\|T_2(r_0)\Phi-T_1(\pm\pi/2)T_2(r+r_0)\Phi\|_{\Hh^2} -\|v-T_1(\theta_0)T_2(r_0)\Phi\|^2_{\Hh^2}\\
&\geq \|\Phi-T_1(\pm\pi/2)T_2(r)\Phi\|_{\Hh^2}-\rho\\
&\geq2R-\rho\\
&>\rho.
\end{split}
\end{equation*}
This means that $\Lambda$ is contained in the strip
$$
\{(\theta,r)\in\R^2;\; |\theta-\theta_0|\leq\pi/2\}.
$$

On the other hand, from \eqref{l5.1},
\begin{equation*}
\begin{split}
\|v-&T_1(\theta+\theta_0)T_2(r_0\pm\widehat{r})\Phi\|_{\Hh^2}\\
&\geq
\|T_2(r_0)\Phi-T_1(\theta)T_2(r_0\pm\widehat{r})\Phi\|_{\Hh^2} -\|v-T_1(\theta_0)T_2(r_0)\Phi\|^2_{\Hh^2}\\
&\geq \|\Phi-T_1(\theta)T_2(\pm\widehat{r})\Phi\|_{\Hh^2}-\rho\\
&\geq2R-\rho\\
&>\rho.
\end{split}
\end{equation*}
This shows that $\Lambda$ is a subset of $\{(\theta,r)\in\R^2;\; |r-r_0|\leq \widehat{r}\}$ and completes the proof of our Claim.

Next define the function
$f(\theta,r)=\|v-T_1(\theta)T_2(r)\Phi\|^2_{\Hh^2}$. Since $\phi$ is
smooth, it is clear that $f$ is a $C^1$ function on $\R^2$.
Therefore,
$$
f(\theta_1,r_1):=\min_{\overline{\Lambda}}f(\theta,r)\leq f(\theta_0,r_0)<\rho^2.
$$
The continuity of $f$ then implies that $(\theta_1,r_1)$ is an interior point of $\Lambda$. Consequently, $\nabla f(\theta_1,r_1)=0$ is equivalent to \eqref{l4} and the proof of the lemma is completed.
\end{proof}

\begin{definition}\label{lydef}
A function $V:\Hh^2\to\R$ is said to be a Lyapunov function for the orbit $\Omega_\Phi$ if the following properties hold.
\begin{itemize}
\item[(i)] There exists $\rho>0$ such that $V:\Omega_\Phi^\rho\to\R$ is of class $C^2$ and, for all $v\in\Omega_\Phi$,
$$
V(v)=0, \quad\mbox{and}\quad \qquad V'(v)=0.
$$
\item[(ii)] There exists $c>0$ such that, for all $v\in\Omega_\Phi^\rho$,
$$
V(v)\geq c[d(v,\Omega_\Phi)]^2.
$$
\item[(iii)] For all $v\in\Omega_\Phi^\rho$, there hold
$$
\langle V'(v),J v\rangle=\langle V'(v),\partial_xv\rangle=0.
$$
\item[(iv)] If $U(t)$ is a global solution of the Cauchy problem associated with \eqref{hamiltonian} with initial datum $U_0$, then $V(U(t))=V(U_0)$, for all $t\geq0$.
\end{itemize}
\end{definition}

Next, we will show the existence of a Lyapunov function to $\Omega_\Phi$.
As in the proof of Lemma \ref{prinlemma}, we denote by
$S:\Hh^2\to\Hh^2$ the self-adjoint operator defined as $S:=\RR^{-1}G''(\Phi)$.

\begin{lemma}\label{lyalemma1}
There are positive constants $M$ and $\delta$ such that
$$
\big(Sv,v\big)_{\Hh^2}+2M\big(\RR^{-1}\I\Phi,v\big)_{\Hh^2}^2\geq\delta\|v\|^2_{\Hh^2},
$$
for all $v\in\{J\Phi,\Phi'\}^\perp=\{v\in\Hh^2;\,\big(J\Phi,v\big)_{\Hh^2}=\big(\Phi',v\big)_{\Hh^2}=0\}$.
\end{lemma}
\begin{proof}
From \eqref{a10.1}, we infer
\begin{equation}\label{l7}
\big(\RR^{-1}\I\Phi,J\Phi\big)_{\Hh^2}=(\Phi,J\Phi)_{\Ll^2}=0
\end{equation}
and
\begin{equation}\label{l8}
\big(\RR^{-1}\I\Phi,\Phi'\big)_{\Hh^2}=(\Phi,\Phi')_{\Ll^2}=0.
\end{equation}
Let $w:=\RR^{-1}\I\Phi/\|\RR^{-1}\I\Phi\|_{\Hh^2}$. Thus, given any $v\in\{J\Phi,\Phi'\}^\perp$, we can find real constants $a,b$, and $c$ such that
\begin{equation}\label{l9}
v=aw+bJ\Phi+c\Phi'+z,
\end{equation}
with $z\in\{J\Phi,\Phi',\RR^{-1}\I\Phi\}^\perp$. By taking the scalar product in \eqref{l9} with $J\Phi$ and $\Phi'$, respectively, it is easily seen that $b=c=0$. As a consequence,
$$
v=aw+z,
$$
with $a=(v,w)_{\Hh^2}$. Since $z\in\{J\Phi,\Phi',\RR^{-1}\I\Phi\}^\perp$, Lemma \ref{prinlemma} implies
\begin{equation}\label{l10}
\begin{split}
\big(Sv,v\big)_{\Hh^2}&=a^2\big(Sw,w\big)_{\Hh^2}+2a\big(Sw,z\big)_{\Hh^2}+\big(Sz,z\big)_{\Hh^2}\\
&\geq a^2\big(Sw,w\big)_{\Hh^2}+2a\big(Sw,z\big)_{\Hh^2}+\delta\|z\|^2_{\Hh^2}.
\end{split}
\end{equation}
But, from Cauchy-Schwartz and Young's inequalities,
$$
2a\big(Sw,z\big)_{\Hh^2}\leq \frac{\delta}{2}\|z\|^2_{\Hh^2}+\frac{2a^2}{\delta}\|Sw\|^2_{\Hh^2}.
$$
Therefore,
$$
\big(Sv,v\big)_{\Hh^2}\geq a^2\big(Sw,w\big)_{\Hh^2}-\Big(\frac{\delta}{2}\|z\|^2_{\Hh^2}+\frac{2a^2}{\delta}\|Sw\|^2_{\Hh^2}\Big)+\delta\|z\|^2_{\Hh^2}
$$
Let $\sigma:=\|\RR^{-1}\I\Phi\|_{\Hh^2}$. In this way,
$$
\big(\RR^{-1}\I\Phi,v\big)_{\Hh^2}=\sigma\big(w,v)_{\Hh^2}=a\sigma.
$$
Also, choose $M>0$ large enough such that
\begin{equation}\label{l11}
\big(Sw,w\big)_{\Hh^2}-\frac{2}{\delta}\|Sw\|^2_{\Hh^2}+2M\sigma^2\geq\frac{\delta}{2}.
\end{equation}
Note that $M$ does not depend on $v$. Now, from \eqref{l11}, we can write
\begin{equation*}
\begin{split}
\big(Sv,v\big)_{\Hh^2}&+2M\big(\RR^{-1}\I\Phi,v\big)_{\Hh^2}^2\\
&\geq a^2\big(Sw,w\big)_{\Hh^2}-\Big(\frac{\delta}{2}\|z\|^2_{\Hh^2}+\frac{2a^2}{\delta}\|Sw\|^2_{\Hh^2}\Big)+\delta\|z\|^2_{\Hh^2}+2Ma^2\sigma^2\\
&=a^2\Big( \big(Sw,w\big)_{\Hh^2}-\frac{2}{\delta}\|Sw\|^2_{\Hh^2}+2M\sigma^2\Big)+\frac{\delta}{2}\|z\|^2_{\Hh^2}\\
&\geq \frac{\delta}{2}\Big(a^2+\|z\|^2_{\Hh^2}\Big)\\
&=\frac{\delta}{2}\|v\|^2_{\Hh^2}.
\end{split}
\end{equation*}
The last equality follows because $w$ and $z$ are orthogonal in
$\Hh^2$. The proof of the lemma is thus completed.
\end{proof}

 Let us set
$$
q_1=G(\Phi), \qquad q_2=F(\Phi).
$$
Given any positive constant $M$, define $V:\Hh^2\to\R$ by
\begin{equation}\label{l12}
V(v)=G(v)-q_1+M(F(v)-q_2)^2.
\end{equation}

We now prove the main result of this subsection.

\begin{proposition}\label{lyalemma2}
There exists $M>0$ such that the functional defined in \eqref{l12} is a Lyapunov function for the orbit $\Omega_\Phi$.
\end{proposition}
\begin{proof}
Since $E$ and $F$ are smooth conserved quantities of \eqref{hamiltonian} and the Cauchy problem associated with \eqref{hamiltonian} is globally well-posed (see Corollary \ref{globalteo}), it is clear that part (iv) in Definition \ref{lydef} is satisfied and $V$ is of Class $C^2$. Since $V(\Phi)=0$ and the functionals $E$ and $F$ are invariant by the transformations \eqref{l0} and \eqref{l00}, we have $V(v)=0$, for all $v\in\Omega_\Phi$. In addition, because
\begin{equation}\label{l13}
\langle V'(u),v\rangle=\langle G'(u),v\rangle+2M(F(u)-q_2)\langle F'(u),v\rangle
\end{equation}
for all $u,v\in\Hh^2$, and $\Phi$ is a critical point of $G$ (see
\eqref{crit}), it also clear that $V'(\Phi)=0$. By observing that
$T_1(\theta)T_2(r)\Phi$ is also a critical point $G$, it then
follows that $V'(v)=0$ for all $v\in \Omega_\Phi$. Part (i) of
Definition \ref{lydef} is also established for any $\rho>0$.

Taking the advantage that
$$
F(T_1(\theta)v)=F(v), \qquad E(T_1(\theta)v)=E(v)
$$
and
$$
F(T_2(r)v)=F(v), \qquad E(T_2(r)v)=E(v)
$$
for all $\theta,r\in\R$ and $v\in\Hh^2$, by taking the derivatives with respect to $\theta$ and $r$ it is not difficult to see that part (iii) in Definition \ref{lydef} is also satisfied for any $\rho>0$.

Finally, let us check part (ii). From \eqref{l13}, we obtain
$$
\langle V''(u)v,v\rangle=\langle G''(u)v,v\rangle+2M(F(u)-q_2)\langle F''(u)v,v\rangle+2M\langle F'(u),v\rangle^2.
$$
In particular,
\begin{equation*}
\begin{split}
\langle V''(\Phi)v,v\rangle&=\langle G''(\Phi)v,v\rangle+2M\langle F'(\Phi),v\rangle^2\\
&=\big(\RR^{-1}G''(\Phi)v,v)_{\Hh^2}+2M\big(\RR^{-1}F'(\Phi),v)_{\Hh^2}^2\\
&=\big(Sv,v)_{\Hh^2}+2M\big(\RR^{-1}F'(\Phi),v)_{\Hh^2}^2.
\end{split}
\end{equation*}
Now recall that
$$
F(U)=\frac{1}{2}(U,U)_{\Ll^2}=\frac{1}{2}(\RR^{-1}\I U,U)_{\Hh^2}=\frac{1}{2}\langle \I U,U\rangle
$$
which implies that $F'=\I$. Hence,
$$
\langle
V''(\Phi)v,v\rangle=\big(Sv,v)_{\Hh^2}+2M\big(\RR^{-1}\I\Phi,v)_{\Hh^2}^2.
$$
Thus, from Lemma \ref{lyalemma1} we deduce the existence of  positive constants $\delta$ and $M$ such that
\begin{equation}\label{l14}
\langle V''(\Phi)v,v\rangle\geq \delta\|v\|_{\Hh^2}^2,
\end{equation}
for all $v\in\{J\Phi,\Phi'\}^\perp$. Since $V$ is of class $C^2$, a Taylor expansion gives
$$
V(v)=V(\Phi)+\langle V'(\Phi),v-\Phi\rangle+\frac{1}{2}\langle V''(\Phi)(v-\Phi),v-\Phi\rangle+h(v),
$$
where $h$ is a function satisfying
$$
\lim_{v\to\Phi}\frac{h(v)}{\|v-\Phi\|^2_{\Hh^2}}=0.
$$
Thus, if $R$ is the constant appearing in Lemma \ref{mlema}, without loss of generality, we can select $\rho\in(0,R/2)$ such that
\begin{equation}\label{l15}
|h(v)|\leq\frac{\delta}{4}\|v-\Phi\|^2_{\Hh^2},  \quad \mbox{for all}\, v\in B_\rho(\Phi).
\end{equation}
By noting that $V(\Phi)=0$ and $V'(\Phi)=0$, and using \eqref{l14} and \eqref{l15}, it follows that
\begin{equation}\label{l16}
\begin{split}
V(v)&=\frac{1}{2}\langle V''(\Phi)(v-\Phi),v-\Phi\rangle+h(v)\\
&\geq \frac{\delta}{2}\|v-\Phi\|^2_{\Hh^2}-\frac{\delta}{4}\|v-\Phi\|^2_{\Hh^2}\\
&=\frac{\delta}{4}\|v-\Phi\|^2_{\Hh^2}\\
&\geq \frac{\delta}{4}[d(v,\Omega_\Phi)]^2,
\end{split}
\end{equation}
provided that $\|v-\Phi\|_{\Hh^2}<\rho$ and $v-\Phi\in\{J\Phi,\Phi'\}^\perp$.

Now take any $v\in\Omega_\Phi^\rho$. Since $\rho<R/2<R$, from Lemma
\ref{mlema} there exist $\theta_1,r_1\in\R$ such that
$u:=T_1(-\theta_1)T_2(-r_1)v\in B_\rho(\Phi)$ and
$$
\big(v-T_1(\theta_1) T_2({r_1})\Phi,J T_1(\theta_1)
T_2({r_1})\Phi)\big)_{\Hh^2} = \big(v-T_1(\theta_1)
T_2({r_1})\Phi,T_1(\theta_1) T_2({r_1})\Phi')\big)_{\Hh^2}=0,
$$
which mean that
$\|u-\Phi\|_{\Hh^2}<\rho$ and $u-\Phi\in\{J\Phi,\Phi'\}^\perp$. Consequently, \eqref{l16} implies
$$
V(v)=V(u)\geq \frac{\delta}{4}[d(u,\Omega_\Phi)]^2=\frac{\delta}{4}[d(v,\Omega_\Phi)]^2.
$$
This proves part (ii) and completes the proof of the lemma.
\end{proof}

\subsection{Orbital stability in the energy space}

Now we prove our main orbital stability result.

\begin{theorem}\label{stateo}
Let $\alpha=4/25$. Let $\phi$ be the solution of \eqref{soleq} given in \eqref{solitary}. Then, the standing wave
$$
\Theta(x,t)=\left(\begin{array}{c}
\phi(x)\cos(\alpha t) \\
\phi(x)\sin(\alpha t)
\end{array}\right)
$$
is orbitally stable in $\Hh^2$.
\end{theorem}
\begin{proof}
The proof is similar to that of Proposition 4.1 in \cite{st}. For the sake of completeness we bring some details. Fix $\ve>0$ and let $V:\Omega_\Phi^\rho\to\R$ be the Lyapunov function given in Proposition \ref{lyalemma2}. Since $V(\Phi)=0$ and $V$ is continuous, there exists $\delta\in(0,\rho)$ such that
$$
V(v)=V(v)-V(\Phi)<c\min\left\{\frac{\rho^2}{4},\ve^2 \right\}, \qquad v\in B_\delta(\Phi),
$$
where $c>0$ is the constant appearing in Definition \ref{lydef}.
The invariance of $V$ with respect to the symmetries in \eqref{l0} and \eqref{l00} then yields
\begin{equation}\label{e1}
V(v)<c\min\left\{\frac{\rho^2}{4},\ve^2 \right\}, \qquad v\in \Omega_\Phi^\delta.
\end{equation}
Let $U_0\in \Hh^2$ be a function such that $U_0\in B_\delta(\Phi)$. From Corollary \ref{globalteo}, the solution, say $U(t)$, of the Cauchy problem associated to \eqref{hamiltonian} with initial data $U_0$ is defined for all $t\geq0$. Let $I$ be the interval defined as
$$
I=\{s>0; \, U(t)\in \Omega_\Phi^\rho \,\,\mbox{for\,\,all}\,\,t\in[0,s)\}.
$$
 The fact that $\delta\in(0,\rho)$ and the continuity of $U(t)$ implies that $I$ is a non-empty interval with $\inf I=0$. We want to show that $I$ is the whole positive semi-line, that is, $s^*:=\sup I=\infty$. Assume by contradiction that $s^*<\infty$. Parts (ii) and (iv) in Definition \ref{lydef} gives
$$
c[d(U(t),\Omega_\Phi)]^2\leq V(U(t))=V(U_0)<c\frac{\rho^2}{4},
$$
for all $t\in[0,s^*)$, where in the last inequality we have used the fact that $U_0\in B_\delta(\Phi)$ and \eqref{e1}. Thus, we deduce that $d(U(t),\Omega_\Phi)<\rho/2$ for all $t\in[0,s^*)$. It is clear that the continuity of $U(t)$ implies the continuity of the function $t\mapsto d(U(t),\Omega_\Phi)$. Consequently, $d(U(s^*),\Phi)\leq\rho/2$. The continuity of $U(t)$ implies again that $\sup I>s^*$, which is a contradiction. Therefore, $I=[0,\infty)$ and
$$
c[d(U(t),\Omega_\Phi)]^2\leq V(U(t))=V(U_0)<c\ve^2
$$
for all  $t\geq0$. The proof of the theorem is thus completed.
\end{proof}

\section*{Acknowledgments}

F. Natali is partially supported by Funda\c c\~ao Arauc\'aria/Paran\'a/Brazil and CNPq/Brazil. A. Pastor is
partially supported by CNPq/Brazil and FAPESP/Brazil.

\end{document}